\documentclass[12pt, a4paper, reqno]{amsart}

\usepackage{enumerate}
\usepackage{amssymb}
\usepackage{color}
\usepackage{graphics}
\usepackage{epsfig}

\newtheorem{lemma}{Lemma}[section]
\newtheorem{corollary}[lemma]{Corollary}
\newtheorem{theorem}[lemma]{Theorem}
\newtheorem{proposition}[lemma]{Proposition}
\newtheorem{remark}[lemma]{Remark}

\newcommand{\R}{\mathbb{R}}

\author[J. D. Garc\'{\i}a-Salda\~{n}a]{Johanna D. Garc\'{\i}a-Salda\~{n}a}
\address{Dept. de Matem\`{a}tiques \\
Universitat Aut\`{o}noma de Barcelona \\ Edifici C. 08193
Bellaterra, Barcelona. Spain} \email{johanna@mat.uab.cat}

\author[A. Gasull]{Armengol Gasull}
\address{Dept. de Matem\`{a}tiques.
Universitat Aut\`{o}noma de Barcelona. Edifici C. 08193 Bellaterra,
Barcelona. Spain} \email{gasull@mat.uab.cat}

\date{}
\dedicatory{} \commby{}

\begin{document}

\title[Harmonic Balance Method ]
{A theoretical basis for\\ the Harmonic Balance Method }
\begin{abstract}
The Harmonic Balance method provides a heuristic approach for
finding truncated Fourier series as an approximation to the periodic
solutions of ordinary differential equations. Another natural way
for obtaining these type of approximations consists in applying
numerical methods. In this paper we recover the pioneering results
of Stokes and Urabe that provide a theoretical basis for proving
that near these truncated series, whatever is the way they have been
obtained, there are actual periodic solutions of the equation. We
will restrict our attention to one-dimensional non-autonomous
ordinary differential equations and we apply the results obtained to
a couple of concrete examples coming from planar autonomous systems.
\end{abstract}

\maketitle

\section{Introduction and main results}

Consider the real non-autonomous differential equation
\begin{equation}\label{sis}
{x}'=X(x,t),
\end{equation}
where the prime denotes the derivative with respect to $t$,
$X:\Omega\times[0,2\pi]\rightarrow\R$ is a $\mathcal{C}^2$-function,
$2\pi$-periodic in $t$, and $\Omega\subset\R$ is a given open
interval.

There are several methods for finding approximations to the periodic
solutions of \eqref{sis}. For instance, the Harmonic Balance method
(HBM), recalled in subsection~\ref{bal}, or simply the numerical
approximations of the solutions of the differential equations. In
any case, from all the methods we can get a truncated Fourier
series, namely a trigonometric polynomial, that ``approximates'' an
actual periodic solution of the equation. The aim of this work is to
recover some old results of Stokes and Urabe that allow to use these
approximations to prove that near them there are actual periodic
solutions and also provide explicit bounds, in the infinity norm, of
the distance between both functions. To the best of our knowledge
these results are rarely used in the papers dealing with HBM.

When the methods are applied to concrete examples one has to deal
with the coefficients of the truncated Fourier series that are
rational numbers (once some number of significative digits is fixed,
see the examples of Section~\ref{exx}) that make more difficult the
subsequent computations. At this point we introduce in this setting
a classical tool, that as far as we know has never been used in this
type of problems: we approximate all the coefficients of the
truncated Fourier series by suitable convergents of their respective
expansions in continuous fractions. This is done in such a way that
using these new coefficients we obtain a new approximate solution
that is essentially at the same distance to the actual solution that
the starting approximation. With this method we obtain trigonometric
polynomials with nice rational coefficients that approximate the
periodic solutions.

Before stating our main result, and following~\cite{Stokes,Urabe},
we introduce some concepts. Let $\bar{x}(t)$ be a $2\pi$-periodic
$\mathcal{C}^1$-function,  we will say that $\bar{x}(t)$ is {\it
noncritical} with respect to \eqref{sis} if
\begin{equation}\label{delta}
\int_0^{2\pi}\frac{\partial}{\partial x}X(\bar{x}(t),t)\,dt\neq 0.
\end{equation}
Notice that if $\bar{x}(t)$ is a periodic solution of \eqref{sis}
then the concept of noncritical is equivalent to the one of being
{\it hyperbolic}, see~\cite{Lloyd}.

As we will see in Lemma~\ref{Le1}, if $\bar x(t)$ is noncritical
w.r.t. equation~\eqref{sis}, the linear periodic system
$$
{y}'=\frac{\partial }{\partial x}X(\bar{x}(t),t)\,y+b(t),
$$
has a unique periodic solution $y_b(t)$ for each smooth
$2\pi$-periodic function $b(t)$. Moreover, once $X$ and $\bar x$ are
fixed, there exists a constant $M$ such that
\begin{equation}\label{norm}
||y_b||_{_\infty}\leq M||b||_{2},
\end{equation}
where as usual, for a continuous $2\pi$-periodic function $f$,
\[
||f||_{_2}=\sqrt{\frac{1}{2\pi}\int_0^{2\pi} f^2(t)dt},\quad
||f||_{_\infty}=\max_{x\in\R}|f(x)| \quad\mbox{and}\quad
||f||_{_2}\leq ||f||_{_\infty} .\] Any constant
satisfying~\eqref{norm} will be called a {\it deformation constant
associated to~$\bar x $ and~$X$}. Finally, consider
\begin{equation}\label{ss}
s(t):=\bar{x}'(t)-X(\bar{x}(t),t).
\end{equation}
 We will say that $\bar{x}(t)$ is
{\it an approximate solution of \eqref{sis} with accuracy
$S=||s||_{_2}$}. For simplicity, if $\tilde{S}>S$ we also will say
that $\bar{x}(t)$ has accuracy $\tilde{S}$. Notice that actual
periodic solutions of~\eqref{sis} have accuracy~0, in this sense,
the function $s(t)$ measures how far is $\bar{x}(t)$ of being an
actual periodic solution of~\eqref{sis}.

Next theorem improves some of the results of Stokes~\cite{Stokes}
and Urabe~\cite{Urabe} in the one-dimensional setting. More
concretely, in those papers they prove the existence and uniqueness
of the periodic orbit when $4M^2KS<1$. We present a similar proof
with the small improvement $2M^2KS<1$. Moreover our result gives,
under an additional condition, the hyperbolicity of the periodic
orbit.

\begin{theorem}\label{T1} Let $\bar{x}(t)$ be
a $2\pi$-periodic $\mathcal{C}^1$--function such that:
\begin{itemize}
\item[-] it is noncritical w.r.t. equation~\eqref{sis} and has $M$ as a
deformation constant,
\item[-] it has accuracy $S$ w.r.t. equation~\eqref{sis}.
\end{itemize}
Given $I:=[\min_{\{t\in\R\}}
\bar{x}(t)-2MS,\max_{\{t\in\R\}}\bar{x}(t)+2MS]\subset\Omega$, let
$K<\infty$ be a constant such that
\[\max_{(x,t)\in I\times [0,2\pi]}\left|\frac{\partial^2 }{\partial
x^2}X(x,t)\right|\le K.
\]
Then, if
\[
2M^2KS<1,
\]
there exists a $2\pi$-periodic solution $x^*(t)$ of~\eqref{sis}
satisfying
\[||x^*-\bar{x}||_{_\infty}\leq 2MS\]
and it is the unique periodic solution of the equation entirely
contained in this strip. If in addition,
\[
\left|\int_0^{2\pi}\frac{\partial}{\partial
x}X(\bar{x}(t),t)\,dt\right|> \frac{2\pi}M,
\]
then the periodic orbit $x^*(t)$ is hyperbolic and its stability is
given by the sign of this integral.
\end{theorem}

Once some approximate solution is guessed, for applying Theorem~\ref{T1} we need
to compute the three constants appearing in its statement. In general, $K$ and $S$
can be easily obtained. Recall for instance that $||s||_{_2}$, when $s$ is a
trigonometric polynomial, can be computed from Parseval's Theorem. On the other
hand $M$ is much more difficult to be estimated. In Lemma~\ref{cota} we give a
result useful for computing it in concrete cases, that is different from the
approach  used in \cite{Stokes,Urabe,UR}.

Assuming that a non-autonomous differential equation has an
hyperbolic periodic orbit, the results of \cite{Urabe} also
guarantee that, if take a suitable trigonometric polynomial
$\bar{r}(t)$ of sufficiently high degree, we can apply the first
part of Theorem~\ref{T1}. Intuitively, while the value of the
accuracy $S$ goes to zero when we increase the degree of the
trigonometric polynomial, the values $M$ and $K$ remain bounded.
Thus at some moment it holds that $2M^2KS<1$.

In Section~\ref{exx} we apply Theorem~\ref{T1} to localize the limit
cycles and prove its uniqueness, in a given region, and its
hyperbolicity for two planar polynomial autonomous systems. The
first one is considered in Subsection~\ref{ss1} and is a simple
example for which the exact limit cycle is already known. We do our
study step by step to illustrate how the method suggested by
Theorem~\ref{T1} works in a concrete example. In particular we
obtain an approximation $\bar x(t)$ of the periodic orbit by using a
combination between the HBM until order 10 and a suitable choice of
the convergents obtained from the theory of continuous fractions
applied to the approach obtained by the HBM.

The second case corresponds to the rigid cubic system
\begin{equation*}
\begin{array}{lll}
\dot x&=&-y+\frac{x}{10}(1-x-10x^2),\\ \dot
y&=&\phantom{-}x+\frac{y}{10}(1-x-10x^2),
\end{array}
\end{equation*}
that in polar coordinates writes as $\dot
r=r/10-\cos(\theta)r^2/10-\cos^2(\theta)r^3,$ $\dot \theta=1$, or equivalently,
\begin{equation}\label{2aa}
r'=\frac{dr}{dt}=\frac{1}{10}\,r-\frac{1}{10}\cos(t)\,r^2
-\cos^2(t)\,r^3,
\end{equation}
 which has a unique positive periodic orbit, see
also~\cite{GaPrTo}. Notice that we have renamed $\theta$ as $t.$ We prove:

\begin{proposition}\label{ppp}
Consider the periodic function {\small\begin{equation*}
\bar{r}(t)=\frac{4}{9}-\frac{1}{693}\cos(t)-\frac{1}{51}\sin(t)
-\frac{1}{653}\cos(2t)-\frac{1}{45}\sin(2t)-\frac{1}{780}\cos(3t).
\end{equation*}}
Then, the differential equation~\eqref{2aa} has a periodic solution
$r^*(t)$, such that \[ ||\bar{r}-r^*||_{_\infty}\leq 0.042,
\]
which is hyperbolic and stable and it is the only periodic solution
of~\eqref{2aa} contained in this strip.
\end{proposition}

As we will see, in this case we will find computational difficulties
to obtain the order three approximation given by the HBM. So we will
get it first approaching numerically the periodic solution; then
computing, also numerically, the first terms of its Fourier series
and finally using again the continuous fractions approach to
simplify the values appearing in our computations. We also will see
that the same approach works for other concrete rigid systems.

Similar examples for second order differential equations have also been studied
in~\cite{UR}.

\section{Preliminary results}

This section contains some technical lemmas that are useful for
proving Theorem~\ref{T1} and for obtaining in concrete examples the
constants appearing in its statement. We also include a very short
overview of the HBM adapted to our interests. See~\cite{Mi} for a
more general point of view on the HBM.

As usual, given $A\subset\R$, ${\bf 1}_A:\R\rightarrow \R$ denotes
the {\it characteristic function of $A$}, that is, the function
takes the value 1 when $x\in A$ and the value 0 otherwise.

\begin{lemma}\label{Le1}
Let $a(t)$ and $b(t)$ be continuous real $2\pi$-periodic functions.
Consider the non-autonomous linear ordinary differential equation
\begin{equation}
x'=a(t)x+b(t). \label{eqlineal}
\end{equation}
If $A(2\pi)\neq 0$, where $A(t):=\int_0^t a(s)ds$, then for each
$b(t)$ the equation \eqref{eqlineal} has a unique $2\pi$-periodic
solution $x_b(t):=\int_0^{2\pi} H(t,s)b(s)ds,$ where the kernel
$H(t,s)$ is given by the piecewise function
\begin{equation}\label{kernel}
H(t,s)=\frac{e^{A(t)}}{1-e^{A(2\pi)}}\left[e^{-A(s)}{\bf
1}_{[0,t]}(s)+ e^{A(2\pi)-A(s)}{\bf
1}_{[t,2\pi]}(s)\right].\end{equation} Moreover
$||x_b||_{_\infty}\leq 2\pi
\max_{t\in[0,2\pi]}||H(t,\cdot)||_{_2}\,||b||_{_2}$.
\end{lemma}

\begin{proof} Since \eqref{eqlineal} is linear, its general solution is
\begin{equation}\label{soleqlin}
x(t)=e^{A(t)}\left(x_0+\int_0^tb(s)e^{-A(s)}ds\right).
\end{equation}
If we impose that the solution is $2\pi$-periodic, i.e.,
$x(0)=x(2\pi)$, we get
\begin{equation}\label{x0}
x_0=\frac{e^{A(2\pi)}}{1-e^{A(2\pi)}}\int_0^{2\pi}b(s)e^{-A(s)}ds.
\end{equation}
By replacing $x_0$ in \eqref{soleqlin} by the right hand side of
\eqref{x0} we obtain that
\begin{align*}
x_b(t)&=\frac{e^{A(t)}}{1-e^{A(2\pi)}}\left[e^{A(2\pi)}\int_0^{2\pi}b(s)e^{-A(s)}ds
+(1-e^{A(2\pi)})\int_0^{t}b(s)e^{-A(s)}ds\right]\\&=
\frac{e^{A(t)}}{1-e^{A(2\pi)}}\left[e^{A(2\pi)}\int_t^{2\pi}b(s)e^{-A(s)}ds
+\int_0^{t}b(s)e^{-A(s)}ds\right]\\&= \int_0^{2\pi}H(t,s)b(s)ds.
\end{align*}
Therefore the first assertion follows. On another hand, by the
Cauchy-Schwarz inequality,
$$
|x_b(t)|\leq
\sqrt{\int_0^{2\pi}H^2(t,s)ds}\sqrt{\int_0^{2\pi}b^2(s)ds}.
$$
Therefore
$$
||x_b||_{_\infty}\leq
2\pi\max_{t\in[0,2\pi]}||H(t,\cdot)||_{_2}\,||b||_{_2}.
$$
This complete the proof.
\end{proof}

\begin{corollary}\label{ccc} A deformation constant~$M$ associated
to~$\bar x$ and $X$ is
\[
M:= 2\pi \max_{t\in[0,2\pi]}||H(t,\cdot)||_{_2},
\]
where $H$ is given in~\eqref{kernel} with
$A(t)=\int_0^t\frac{\partial }{\partial x}X(\bar{x}(t),t)\,dt.$
\end{corollary}

Now we prove a technical result that will allow us to compute in practice
deformation constants. In fact we will find an upper bound of $M$ that will avoid
the integration step needed in the computation of the norm $||\cdot||_{_2}$.
First, we introduce some notation.

Given a function $A:[0,2\pi]\to \R$, a partition $t_i=ih,
i=0,1,\ldots,N,$ of the interval $[0,2\pi]$, where $h=2\pi/N$, and a
positive number $\ell$, we consider the function $L:[0,2\pi]\to \R$
given by the continuous linear piecewise function joining the points
$(t_i,A(t_i)-\ell)$. Notice that $L(t)=\sum_{i=0}^{N-1}L_i(t){\bf
1}_{I_i}$, where $I_i=[t_i,t_{i+1}]$ and
\[L_i(t)= \frac{A(t_{i+1})-A(t_i)}h (t-t_i)+f(t_i):=-\frac12(\alpha_i
t+\beta_i).\]

We will say that $L$ is {\it an adequate lower bound of $A$} if it
holds that $L(t)<A(t)$ for all $t\in[0,2\pi].$ It is clear that
smooth functions have always adequate functions, that approach to
them.

In next result we will use the following functions
\begin{equation}\label{psi}
\Psi_m(t):=\sum_{i=0}^{m-1} J_i+\lambda^2\sum_{i=m-1}^{N-1}
J_i+(1-\lambda^2)\frac{e^{\beta_m}}{\alpha_m}\left(
e^{\alpha_mt}-e^{\alpha_mt_{m}}\right),
\end{equation}
where
\[
J_i:=\int_{t_i}^{t_{i+1}} e^{-2L(s)}ds=\int_{t_i}^{t_{i+1}}
e^{-2L_i(s)}ds=\frac{e^{\beta_i}}{\alpha_i}\left(
e^{\alpha_it_{i+1}}-e^{\alpha_it_{i}}\right)
\]
and $\lambda=e^{A(2\pi)}$.

\begin{lemma}\label{cota} Let $L$ be an adequate lower bound of $A$, where $A$ is
the function given in Lemma~\ref{Le1}. Consider the functions
$\Psi_m(t), m=0,1,\ldots,N-1$, given in~\eqref{psi}. Therefore,
following also the notation introduced in that Lemma, it holds that
$||x_b||_{_\infty}\leq N\,||b||_{_2}$, where
\[
N= \frac{\sqrt{2\pi}}{|1-\lambda|}
\max_{t\in[0,2\pi]}e^{A(t)}\sqrt{\sum_{m=0}^{N-1} \Psi_m(t){\bf
1}_{I_m}(t)}.
\]
\end{lemma}
\begin{proof}

Recall that from Lemma \ref{Le1}, $||x_b||_{_\infty}\leq
M\,||b||_{_2}$, where
\[
M:= 2\pi \max_{t\in[0,2\pi]}||H(t,\cdot)||_{_2}.
\]
So we will find an upper bound of $M.$ Since
$$
H(t,s)=\frac{e^{A(t)}}{1-e^{A(2\pi)}}\left[e^{-A(s)}{\bf
1}_{[0,t]}(s)+ e^{A(2\pi)-A(s)}{\bf 1}_{[t,2\pi]}(s)\right],
$$
it holds that
$$
||H(t,\cdot)||_{_2}=\frac{1}{\sqrt{2\pi}}\frac{e^{A(t)}}{\left|1-\lambda\right|}
\sqrt{G(t)}
$$
where
$$
G(t):=\int_0^t e^{-2A(s)}ds+\lambda^2\int_t^{2\pi} e^{-2A(s)}ds
<\int_0^t e^{-2L(s)}ds+\lambda^2\int_t^{2\pi} e^{-2L(s)}ds,
$$
because $L(t)<A(t)$, for all $t\in[0,2\pi]$.

Assume that $t\in I_m$. Then

\begin{align*}
\int_0^t e^{-2L(s)}ds&= \sum_{i=0}^{m-1} J_i+\int_{t_m}^t e^{-2L_m(s)}ds \\
\int_t^{2\pi} e^{-2L(s)}ds &= \sum_{i=m}^{N-1}
J_i+\int_{t}^{t_{m+1}} e^{-2L_m(s)}ds= \sum_{i=m-1}^{N-1}
J_i-\int_{t_m}^{t} e^{-2L_m(s)}ds.
\end{align*}

Therefore, for $t\in I_m$,
\[
G(t)<\sum_{i=0}^{m-1} J_i+\lambda^2\sum_{i=m-1}^{N-1}
J_i+(1-\lambda^2)\int_{t_m}^t e^{\alpha_ms+\beta_m}ds=\Psi_m(t).
\]
As a consequence, for $t\in[0,2\pi]$,
\[
G(t)<\sum_{m=0}^{N-1} \Psi_m(t){\bf 1}_{I_m}(t),
\]
and the result follows.
\end{proof}

\begin{remark}\label{cotaM} Notice that the above lemma
provides a way for computing a deformation constant where there is
no need of computing integrals. This will be very useful in concrete
application, where the primitive of $e^{-2 A(t)}$ is not computable
and so Corollary~\ref{ccc} is difficult to apply for obtaining $M.$
\end{remark}

In next result, which introduces the constant $K$ appearing in
Theorem~\ref{T1}, ${D}^\circ$ denotes the topological interior of
$D.$

\begin{lemma}\label{Le2} Consider $X$ as in~\eqref{sis}. Let $D$ be a closed
interval and let $\bar x(t)$ be a $2\pi$-periodic
$\mathcal{C}^1$-function, such that $\{\bar x(t)\,:\,t\in\R
\}\subset {D}^\circ$. Define
\begin{equation}\label{R}
R(z,t):=X(\bar{x}(t)+z,t)-X(\bar{x}(t),t) -\frac{\partial}{\partial
x}X(\bar{x}(t),t)z,
\end{equation}
for all $z$ such that $\{\bar{x}(t)+z\,:\,t\in\R\}\subset D$. Then
\begin{itemize}
\item[$(i)$]$|R(z,t)|\leq \frac{K}{2}|z|^2$,\vspace {0.2cm}
\item[$(ii)$]$|R(z,t)-R(\bar{z},t)|\leq
K\max(|z|,|\bar{z}|)\,|z-\bar{z}|,$
\end{itemize}
where
$$
K:=\max_{(x,t)\in D\times [0,2\pi]}\left| \frac{\partial^2}{\partial
x^2}X(x,t)\right|.
$$
\end{lemma}
\begin{proof}
$(i)$. By using the Taylor's formula, for each $t$ it holds that
$$
X(\bar{x}(t)+z,t)=X(\bar{x}(t),t)+\frac{\partial}{\partial
x}X(\bar{x}(t),t)z+\frac{1}{2}\frac{\partial^2}{\partial
x^2}X(\xi(t),t)z^2
$$
for some $\xi(t)\in\langle\bar{x}(t),\bar{x}(t)+z\rangle$. Therefore
$$
|R(z,t)|=\left|\frac12\frac{\partial^2 }{\partial
x^2}X(\xi(t),t)\right||z|^2\leq \frac{K}{2}|z|^2,
$$
as we wanted to prove.

\smallskip

\noindent $(ii)$. From Rolle's Theorem for each fixed $t$ it follows
that there exists $\eta(t)\in\langle z,\bar{z}\rangle$ such that
\begin{equation*}\label{Roll}
\left|R(z,t)-R(\bar{z},t)\right|\leq \left|\frac{\partial}{\partial
z}R(\eta(t),t)\right||z-\bar{z}|.
\end{equation*}
Applying again this theorem, but now to $\frac{\partial}{\partial
z}R$, noticing that $\left.\frac{\partial}{\partial
z}R(z,t)\right|_{z=0}=0$, we obtain that
\begin{equation*}\label{Roll2}
\left| \frac{\partial}{\partial z} R(\eta(t),t)\right|\leq
\left|\frac{\partial^2}{\partial z^2}R(\omega(t),t)\right||\eta(t)|
=\left|\frac{\partial^2}{\partial
x^2}X(\omega(t),t)\right||\eta(t)|\le K|\eta(t)|,
\end{equation*}
where $\omega(t)\in\langle 0,\eta(t)\rangle$. Note also that
\begin{equation*}\label{desi}
|\eta(t)|\leq\max(|z|,|\bar{z}|).
\end{equation*}
Hence, the result follows combining the three inequalities.
\end{proof}

\subsection{The Harmonic Balance method}\label{bal} In this subsection we recall the HBM adapted to
the setting of one-dimensional $2\pi$-periodic non-autonomous
differential equations.

We are interested in finding periodic solutions of the
$2\pi$-periodic differential equation \eqref{sis}, or equivalently,
periodic functions which satisfy the following functional equation
\begin{equation}\label{fe}
\mathcal{F}(x(t)):=x'(t)-X(x(t),t)=0.
\end{equation}

Recall that any smooth $2\pi$-periodic function $x(t)$ can be
written as its Fourier series,
$$
x(t)= \frac{a_0}2+\sum_{m=1}^{\infty} \left(a_m \cos(m t)+ b_m\sin(m
t)\right),
$$
where
\[
a_{m}=\frac{1}{\pi}\int_0^{2\pi}x(t)\cos(mt)\,dt,\quad\mbox{and}\quad
b_{m}=\frac{1}{\pi}\int_0^{2\pi}x(t)\sin(mt)\,dt,
\]
for all $m\ge0.$ Hence it is natural to try to approach the periodic
solutions of the functional equation \eqref{fe} by using truncated
Fourier series, {\it i.e.} trigonometric polynomials.

Let us describe the HBM of order $N$. Consider a trigonometric
polynomial
$$
y_{_N}(t)=\frac{r_0}2+\sum_{m=1}^{N} \left(r_m \cos(m t)+ s_m\sin(m
t)\right),
$$
with unknowns $r_m=r_m(N),s_m=s_m(N)$ for all $m\le N$. Then compute
the $2\pi$-periodic function $\mathcal{F}(y_{_N}(t))$. It has also
an associated Fourier series
\[
\mathcal{F}(y_{_N}(t))=\frac{\mathcal{A}_0}2+\sum_{m=1}^{\infty}
\left(\mathcal{A}_m \cos(m t)+ \mathcal{B}_m\sin(m t)\right),
\]
where $\mathcal{A}_m=\mathcal{A}_m({\bf r},{\bf s})$ and
$\mathcal{B}_m=\mathcal{B}_m({\bf r},{\bf s})$, $m\ge0,$ with ${\bf
r}=(r_0,r_1,\ldots,r_{_N})$ and ${\bf s}=(s_1,\ldots,s_{_N})$. The
HBM consists in finding values ${\bf r}$ and ${\bf s}$ such that
\begin{equation}\label{nl}
\mathcal{A}_m({\bf r},{\bf
s})=0\quad\mbox{and}\quad\mathcal{B}_m({\bf r},{\bf s})=0\quad
\mbox{for}\quad 0\le m\le N.\end{equation} The above set of
equations is usually a very difficult non-linear system of equations
and for this reason in many works, see for instance \cite{Mi} and
the references therein, only small values of $N$ are considered. We
also remark that in general the coefficients of $y_{_N}(t)$ and
$y_{_{N+1}}(t)$ do not coincide at all.

Notice that equations \eqref{nl} are equivalent to
\[
\int_0^{2\pi} \mathcal{F}(y_{_N}(t))\cos
(mt)\,dt=0\quad\mbox{and}\quad \int_0^{2\pi}
\mathcal{F}(y_{_N}(t))\sin (mt)\,dt=0,
\]
for $0\le m\le N.$

The hope of the method is that the trigonometric polynomials found
using this approach are ``near" actual periodic solutions of the
differential equation \eqref{sis}. In any case, as far as we know,
the BHM for $N$ small is only a heuristic method that sometimes
works quite well.

To end this subsection, we want to comment a main difference between
the non-autonomous case treated here and the autonomous one. In this
second situation the periods of the searched periodic orbits, or
equivalently their frequencies, are also treated as unknowns. Then
the methods works similarly, see again \cite{Mi}.

\section{Proof of the main result}

\begin{proof}[Proof of Theorem~\ref{T1}]

As a first step we prove the following result: consider the
nonlinear differential equation
\begin{equation}\label{zp}
{z}'=X(z+\bar{x}(t),t)-X(\bar{x}(t),t)-s(t),
\end{equation}
where $s(t)$ is given in~\eqref{ss}. Then a $2\pi$-periodic function
$z(t)$ is a solution of~\eqref{zp} if and only if $z(t)+\bar{x}(t)$
is a $2\pi$-periodic solution of \eqref{sis}.

This is a consequence of the following equalities
\begin{align*}
(z(t)+\bar x(t))'=&[X(z(t)+\bar{x}(t),t)-X(\bar{x}(t),t)-s(t)]
+[X(\bar{x}(t),t)+s(t)]\\
=&X(z(t)+\bar{x}(t),t).
\end{align*}

By using the function
\[
R(z,t)=X(z+\bar{x}(t),t)-X(\bar{x}(t),t)-\frac{\partial}{\partial
x}X(\bar{x}(t),t)z,
\]
introduced in Lemma~\ref{Le2}, equation~\eqref{zp} can be written as
\begin{equation}\label{zp2}
{z}'=\frac{\partial}{\partial x}X(\bar{x}(t),t)z+R(z,t)-s(t).
\end{equation}

Let ${\mathcal P}$ be the space of $2\pi$-periodic
$\mathcal{C}^0$-functions. To prove the first part of the theorem it
suffices to see that equation~\eqref{zp2} has a unique
$\mathcal{C}^1$, $2\pi$-periodic solution $z^*(t)$, which belongs to
the set
$$\mathcal{N}=\{z\in\mathcal{P}: ||z||_{_\infty}\leq 2MS\}.$$

To prove this last assertion we will construct a contractive map
$T:\mathcal{N}\rightarrow\mathcal{N}$. Because $\mathcal{N}$ is a complete space
with the $||\cdot||_{_\infty}$ norm, its fixed point will be a continuous function
in $\mathcal{N}$ that will satisfy an integral equation, equivalent
to~\eqref{zp2}. Finally we will see that this fixed point is in fact a
$\mathcal{C}^1$ function and that it satisfies equation~\eqref{zp2}.

Let us define $T$. If $z\in\mathcal{N}$ then $T(z)$ is defined as
the unique $2\pi$-periodic solution of the linear differential
equation
$$
{y}'=\frac{\partial}{\partial x}X(\bar{x}(t),t)y+R(z(t),t)-s(t).
$$
Notice that this map is well defined, by Lemma~\ref{Le1}, because
$\bar{x}(t)$ is noncritical w.r.t. equation~\eqref{sis}. Then $z_1$
satisfies
$$
{z}'_1=\frac{\partial}{\partial
x}X(\bar{x}(t),t)z_1+R(z(t),t)-s(t).$$

Let us prove that $T$ maps $\mathcal{N}$ into $\mathcal{N}$ and that
it is a contraction. By Lemmas \ref{Le1} and \ref{Le2} and the
hypotheses of the theorem
\begin{align*}
||T(z)||_{_\infty}=||z_1||_{_\infty} &\leq
M||R(z(\cdot),\cdot)-s(\cdot)||_{_2} \leq
M\left(||R(z(\cdot),\cdot)||_{_2}+S\right)\\
&\leq M(||R(z(\cdot),\cdot)||_{_\infty}+S) \leq
M(\frac{K}{2}||z||_{_\infty}^2+S)\\
&\leq M(2KM^2S^2+S)< 2MS,
\end{align*}
where we have used in the last inequality that $2M^2KS<1$.

To show that $T$ is a contraction on $\mathcal{N}$, take
$z,\bar{z}\in\mathcal{N}$ and denote by $z_1=T(z)$,
$\bar{z}_1=T(\bar{z})$. Then
\begin{align*}
{z}'_1&=\frac{\partial}{\partial x}X(\bar{x}(t),t)z_1+R(z(t),t)-s(t),\\
{\bar{z}}'_1&=\frac{\partial}{\partial
x}X(\bar{x}(t),t)\bar{z}_1+R(\bar{z}(t),t)-s(t).
\end{align*}
Therefore
$$
(z_1-\bar{z}_1)'=\frac{\partial}{\partial
x}X({\bar{x}}(t),t)(z_1-\bar{z}_1)+R(z(t),t)-R(\bar{z}(t),t).
$$
Again by Lemmas \ref{Le1} and \ref{Le2} and the hypotheses of the
theorem,
\begin{align*}
||T(z)-T(\bar{z})||_{_\infty}=&||z_1-\bar{z}_1||_{_\infty}\leq
M||R(z(\cdot),\cdot)-R(\bar{z}(\cdot),\cdot)||_{_\infty}\\
\le&
MK\max(||z||_{_\infty},||\bar{z}||_{_\infty})||z-\bar{z}||_{_\infty}\leq
2M^2KS||z-\bar{z}||_{_\infty},
\end{align*}
as we wanted to prove, because recall that $2M^2KS<1$.

Therefore the sequence of functions $\{z_n(t)\}$ defined as
$$
{z}'_{n+1}(t)=\frac{\partial}{\partial
x}X(\bar{x}(t),t)z_{n+1}(t)+R(z_n(t),t)-s(t),$$ with any
$z_0(t)\in\mathcal{N}$, and $z_{n+1}(t)$ chosen to be periodic,
converges uniformly to some function $x^*(t)\in\mathcal{N}$. In fact
we also have that
$$
{z}_{n+1}(t)=z_{n+1}(0)+\int_0^t \left(\frac{\partial}{\partial
x}X(\bar{x}(w),w)z_{n+1}(w)+R(z_n(w),w)-s(w)\right)\,dw.$$ Therefore
\[
x^*(t)=x^*(0)+\int_0^t \left(\frac{\partial}{\partial
x}X(\bar{x}(w),w)x^*(w)+R(x^*(w),w)-s(w)\right)\,dw.
\]
We know that $x^*(t)$ is a continuous function, but from the above
expression we obtain that it is indeed of class~ $\mathcal{C}^1$.
Therefore $x^*(t)$ is a periodic solution of~\eqref{zp2} and is the
only one in $\mathcal N$, as we wanted to see.

To prove the hyperbolicity of $x^*(t)$ it suffices to show that
$$
\int_0^{2\pi}\frac{\partial}{\partial x}X(x^*(t),t)dt\neq 0,
$$
and study its sign, see \cite{Lloyd}. We have that, fixed $t$,
$$
\frac{\partial}{\partial x}X(x^*(t),t)=\frac{\partial}{\partial
x}X(\bar{x}(t),t)+\frac{\partial^2}{\partial
x^2}X(\xi(t),t)(x^*(t)-\bar{x}(t)),
$$
for some $\xi(t)\in\langle x^*(t),\bar{x}(t)\rangle$. Therefore,
since we have already proved that $|x^*(t)-\bar x(t)|<2MS$,
$$
\left|\frac{\partial}{\partial
x}X(\bar{x}(t),t)-\frac{\partial}{\partial
x}X(x^*(t),t)\right|\leq2KMS.
$$
Then
$$\left|\int_0^{2\pi}\frac{\partial}{\partial
x}X(\bar{x}(t),t)dt-\int_0^{2\pi}\frac{\partial}{\partial
x}X(x^*(t),t)dt\right|\leq 4\pi KMS<\frac{2\pi}{M}
$$
and the results follows because by hypothesis the first integral is,
in absolute value, bigger that $2\pi/M.$
\end{proof}

\section{Applications}\label{exx}

In this section we apply our result to prove the existence and
localize a hyperbolic limit cycle of some planar systems, which
after some transformations can be converted into differential
equations of the form~\eqref{sis}. In the first case, although we
know explicitly the limit cycle, we first use the HBM to approximate
it and then Theorem~\ref{T1} to prove in an alternative way its
existence. In the second case we consider a planar rigid system.
First, we found numerically an approximation of the limit cycle and
from this approximation we propose a truncated Fourier series as a
simpler approximation. Finally, Theorem~\ref{T1} is used again to
prove the existence and localize the limit cycle.

\subsection{A simple integrable case}\label{ss1}
Consider the planar ordinary differential equation
\begin{equation}\label{ex1}
\begin{array}{lll}
\dot x&=&-y+x(a+dx^2+exy+fy^2)\\
\dot y&=&x+y(a+dx^2+exy+fy^2)
\end{array}
\end{equation}
In polar coordinates it writes as
$$\dot r=ar+(d\cos^2(\theta)+e\sin(\theta)\cos(\theta)+f\sin^2(\theta))r^3,
\quad \dot\theta=1,$$ or equivalently,
$$
r'=\frac{dr}{dt}=ar+(d\cos^2(t)+e\sin(t)\cos(t)+f\sin^2(t))r^3:=X(r,t),
$$
where we have renamed $\theta$ as $t.$ The above equation is a Bernoulli equation
that can be solved explicitly. For simplicity we fix $a=-1$, $d=3$, $e=2$ and
$f=1$. Then we have the equation
\begin{equation}\label{1a}
\dot{r}=-r+(\cos(2t)+\sin(2t)+2)r^3.
\end{equation}
Its solutions are $r(t)\equiv 0$ and
$$
r(t)=\pm\frac{1}{\sqrt{2+\cos(2t)+ke^{2t}}}.
$$
Therefore its unique positive periodic solution, which corresponds
to the only limit cycle of \eqref{ex1} for the given values of the
parameters, is given by the ellipse
\begin{equation}
r^*(t)=\frac{1}{\sqrt{2+\cos(2t)}}. \label{solr}
\end{equation}
Moreover since
$$
\int_0^{2\pi}\frac{\partial}{\partial r}X(r^*(t),t)\,dt=4\pi>0
$$
it is hyperbolic and unstable, see \cite{Lloyd}. Its Fourier series
is
\begin{equation}
r^*(t)=\frac{a_0}{2}+\sum_{k=1}^{\infty}
a_{2k}\cos(2t),\label{seFou}
\end{equation}
where
$$
\begin{array}{lll}
a_0&=&\frac{4 K}{\sqrt3\pi} \approx
1.491498374, \quad a_0/2\approx 0.745749187, \\\\
a_2&=&\frac{12 E -8 K }{\sqrt3\pi}\approx
-0.2016837219,\\\\
a_4&=&\frac{-32 E +20 K }{\sqrt3\pi}\approx
0.04065713288,\\\\
a_6&=&\frac{476 E -296 K }{\sqrt3\pi}\approx
-0.009092598292,\\\\
a_8&=&\frac{-10624 E +6604 K }{\sqrt3\pi}\approx
0.002133790322,\\\\
a_{10}&=&\frac{105548 E -65608 K }{\sqrt3\pi}\approx
-0.0005148662408,
\end{array}
$$
being $K=K(\sqrt 6/3)$ and $E=E(\sqrt 6/3)$ the complete elliptic
integrals of the first and second kind respectively, see \cite{BF}.

Let us forget that we know the exact solution and its full Fourier
series to illustrate how to use the HBM and Theorem~\ref{T1} for
equation~\eqref{1a} to obtain an approach to the actual periodic
solution~\eqref{solr}.

Following the HBM, see subsection \ref{bal}, consider the equation
\begin{equation}\label{11a}
\mathcal{F}(r(t))=r'(t)+r(t)-(\cos(2t)+\sin(2t)+2)r^3(t)=0,
\end{equation}
which is clearly equivalent to~\eqref{1a}.

Searching for a solution of the form $r(t)=r_0$ and imposing that
the first harmonic of $\mathcal{F}(r(t))$ vanishes we get that
$r_0+2r_0^3=0.$ The only positive solution of the equation is
$r_0=\sqrt{2}/2\approx0.7071$ and this is the first order solution
given by HBM.

Motivated by the symmetries of \eqref{1a} for applying the second
order HBM we search for an approximation of the form
$$
r(t)=r_0+r_2\cos(2t).
$$
The vanishing of the coefficients of 1 and $\cos(2t)$ in the Fourier
series of $\mathcal{F}(r(t))$ give the non-linear system:
\begin{align*}
g(r_0,r_2)&:=r_0-2r_0^3-\frac{3}{2}r_2r_0^2-3r_2^2r_0-\frac{3}{8}r_2^3=0,\\
h(r_0,r_2)&:=r_2-r_0^3-6r_2r_0^2-\frac{9}{4}r_2^2r_0-\frac{3}{2}r_2^3=0.
\end{align*}
Doing the resultants $\operatorname{Res}(g,h,r_0)$,
$\operatorname{Res}(g,h,r_2)$ we obtain that the solutions of the
above system are also solutions of
\begin{align*}
&219720r_0^8-18852r_0^6+4269r_0^4-328r_0^2+8=0,\\
&49437r_2^8-70956r_2^6+30708r_2^4-4288r_2^2+128=0.
\end{align*}
One of its solutions is $r_0\approx 0.7440456581=:\tilde{r}_0$,
$r_2\approx -0.2013905597=:\tilde{r}_2$.

To know the accuracy of the periodic function
$\tilde{r}(t)=\tilde{r}_0+\tilde{r}_2\cos(2t)$ as a solution
of~\eqref{1a} we compute
$$
\widetilde{S}=||\tilde{r}'(t)+\tilde{r}(t)-(2+\sin(2t)+
\cos(2t))\tilde{r}(t)^3||_{_2}\approx 0.1361
$$
Since it is enough for our purposes we can consider simpler rational
approximations of $\tilde{r}_0$ and $\tilde{r}_1$, but keeping a
similar accuracy. For finding these rational approximations, we
search them doing the continuous fraction expansion of these values.
For instance
$$
\tilde{r}_0=[0,1,2,1,9,1,21,17,3,11]
$$
giving the convergents $1$, $2/3$, $3/4$, $29/39$, $32/43$,$\ldots$.
Similarly $\tilde{r}_2$ gives $1/4$, $1/5$, $28/139$,
$29/144$,$\ldots$. At this point we have the following new candidate
to be an approximation of the periodic solution
$$
\bar{r}(t)=\frac{3}{4}-\frac{1}{5}\cos(2t).
$$
Its accuracy w.r.t. equation~\eqref{1a} is \[
S=||\bar{r}'(t)+\bar{r}(t)-(2+\sin(2t)+
\cos(2t))\bar{r}(t)^3||_{_2}= \frac{\sqrt{50069}}{1600}\approx
0.1398<0.14,\] and so, quite similar to the one of $\tilde{r}(t).$

Therefore $\tilde{r}(t)$ and $\bar{r}(t)$ are solutions of
\eqref{1a} with similar accuracy so we keep $\bar{r}(t)$ as the
second order approximation given by this modification of the HBM.
For this $\bar{r}(t)$ we already know that its accuracy is $S=0.14$.

We need to know the value of $M$ given in Theorem \ref{T1}. With
this aim we will apply Lemma \ref{cota}. We consider in that lemma a
function $L(t)$ formed by $13$ straight lines and $\ell=1/9$. Then
we get that we can take $M=2.3$. Therefore, since $2MS=0.644$ and
$0.55=\frac{11}{20}\leq\bar{r}(t)\leq\frac{19}{20}=0.95$.

\smallskip

We have that $I=[-0.094,1.594]$ in Theorem \ref{T1}. Moreover
$$
\left|\frac{\partial^2}{\partial r^2}X(r,t)\right|\leq
6|2+\sin(2t)+\cos(2t)||r|\leq (12+6\sqrt{2})|r|\leq \frac{41}{2}|r|
$$
Thus taking $K=\frac{41}{2}(1.594)\approx 32.68$ we get that
$2M^2KS\approx 48.4>1$ and we can not apply Theorem \ref{T1}.

\smallskip

Doing similar computations with the successive approaches given by
the HBM we obtain
$$
\begin{array}{lll}
\bar{r}(t)&=&\frac{3}{4}-\frac{1}{5}\cos(2t)+\frac{1}{25}\cos(4t),\\\\
\bar{r}(t)&=&\frac{3}{4}-\frac{1}{5}\cos(2t)+\frac{1}{25}\cos(4t)-
\frac{1}{110}\cos(6t).
\end{array}
$$
It is worth to comment that the above two functions are periodic
functions that approximate to solution of \eqref{1a} with accuracies
$0.045$ and $0.018$, respectively, while the solutions obtained
solving approximately the non-linear systems with ten significative
digits have similar accuracies, namely $0.043$ and $0.013$,
respectively. For none of both approaches Theorem \ref{T1} applies.
Let us see that the next order HBM works for this example.

If we do all the computations we obtain the candidate to be solution
$$
\tilde{r}(t)=\sum_{k=0}^4 r_{2k}\cos(2kt),
$$
with
\begin{align*}
r_0=&\,\,0.7457489122,\qquad\quad\,\, r_2=-0.2016836610,\quad r_4=0.04065712547,\\
r_6=&-0.009092599917,\quad\,\, r_8=0.002133823488.
\end{align*}
Computing the accuracy of $\tilde{r}(t)$ we obtain that it is
$0.0039$. If we take the approximation, using some convergents of
$r_{2k}$,
$$
\bar{r}(t)=\frac{3}{4}-\frac{1}{5}\cos(2t)+
\frac{1}{25}\cos(4t)-\frac{1}{110}\cos(6t)+\frac{1}{468}\cos(8t)
$$
it has accuracy $0.0125$. This means that we have lost significative
digits and we need to take convergents of $r_{2k}$ that have at
least 3 significative digits. For instance some convergents of $r_0$
are $1$, $2/3$, $3/4$, $41/55$, $44/59$,$\ldots$ and we choose
$44/59$. Finally we consider
\begin{equation}\label{af}
\bar{r}(t)=\frac{44}{59}-\frac{24}{119}\cos(2t)+
\frac{2}{49}\cos(4t)-\frac{1}{110}\cos(6t)+\frac{1}{468}\cos(8t).
\end{equation}
The accuracy of $\bar{r}$ is $0.00394$ quite similar to the one of
$\tilde{r}(t)$. So we take $S=0.004$. Let us see that Theorem
\ref{T1} applies if we take this approximate periodic solution.

In this case, by applying Lemma \ref{cota}, using the piecewise
linear function $L$ formed by $10$ pieces and $\ell=1/10$, we obtain
that we can take $M=2.4.$

\smallskip

Since it can be seen that $0.5\leq\bar{r}(t)\leq 1$ and $2MS=0.0192$
we can take in Theorem \ref{T1} the interval $I:=[0.4808, 1.0192]$.

Then
$$
\max_{I\times[0,2\pi]}\left|\frac{\partial^2}{\partial
r^2}X(r,t)\right|\leq \frac{41}{2}(1.02)=20.91=:K.
$$
Finally, $2M^2KS\approx 0.96<1$ and Theorem \ref{T1} applies.

Finally, it is easy to see that
$$
\int_0^{2\pi}\frac{\partial}{\partial r}X(\bar{r}(t),t)dt>12.5,
$$
which is bigger than $2\pi/M \approx 2.6$. Therefore the
hyperbolicity of the periodic orbit given by Theorem~\ref{T1}
follows. In short we have proved,
\begin{proposition}
Consider the periodic function $\bar{r}(t)$ given in \eqref{af}.
Then there is a periodic solution $r^*(t)$ of \eqref{1a}, such that
$$
||\bar{r}-r^*||_{_\infty}\leq 0.0192,
$$
which is hyperbolic and unstable and it is the only periodic
solution of \eqref{1a} in this strip.
\end{proposition}

\begin{remark}
Using the known analytic expression of $r^*(t)$ it can be seen that
indeed
$$
||\bar{r}-r^*||_{_\infty}\leq 0.0007.
$$
\end{remark}

Notice that by using a high enough HBM we have obtained a proof of
the existence of a hyperbolic periodic orbit and an effective
approximation $\bar{r}(t)$ without integrating the differential
equation.

\subsection{A rigid cubic system}\label{ss2}
In this section we study some concrete cases of the family of rigid
cubic systems
\begin{equation}\label{ex2f}
\begin{array}{lll}
\dot x&=&-y-x(a+bx+x^2),\\
\dot y&=&\phantom{-}x-y(a+bx+x^2),
\end{array}
\end{equation}
already considered in~\cite{GaPrTo}. In that paper it is proved that
\eqref{ex2f} has at most one limit cycle and when it exists is
hyperbolic. With our point of view we will find an explicit
approximation of the limit cycle, see Proposition~\ref{ppp}. We
consider the case $a=-b=1/10$, that in polar coordinates writes
as~\eqref{2aa},
\begin{equation*}
r'=\frac{dr}{dt}=\frac{1}{10}\,r-\frac{1}{10}\cos(t)\,r^2
-\cos^2(t)\,r^3,
\end{equation*}
and we start explaining how we have found the approximation of the
periodic solution of~\eqref{2aa} given in Proposition~\ref{ppp}.

\smallskip

\noindent{\bf First attempt: the HBM}. First we try to apply this
method to find an approximation of the periodic solution of
\eqref{2aa} that allows to use Theorem~\ref{T1}.

Searching for a solution of the form $r(t)=r_0$ and imposing that
the first harmonic of
$$
\frac{1}{2}r_0^3-\frac{1}{10}r_0+\frac{1}{10}\cos(t) r_0^2
+\frac{1}{2}\cos(2t)r_0^3
$$
vanishes we obtain that
$$\frac{1}{2}r_0\left(r_0^2-\frac{1}{5}\right)=0.$$
Hence $r_0=\sqrt{5}/5\approx 0.4472135954$ is the first order
solution given by the HBM. We obtain that the positive approximate
solution is $r=\sqrt{5}/5$. For applying the second order HBM we
search for an approximation of the form
$$
r(t)=r_0+r_1\cos(t)+s_1\sin(t).
$$
The vanishing of the coefficients of 1, $\cos(t)$ and $\sin(t)$ in
$\mathcal{F}(r(t))$ provides the non-linear system
\begin{align*}
&\frac{9}{4}r_0^2r_1-\frac{5}{8}r_1^3+\frac{3}{8}r_1s_1^2
+\frac{1}{10}r_0^2+\frac{3}{40}r_1^2+\frac{1}{40}s_1^2-\frac{1}{10}
r_1=0,\\
&\frac{3}{4}r_0^2s_1+\frac{3}{8}r_1^2s_1+\frac{1}{8}s_1^3
+\frac{1}{20}r_1s_1-\frac{1}{10}s_1-r_1=0,\\
&\frac{1}{2}r_0^3+\frac{9}{8}r_0r_1^2-\frac{1}{10}r_0
+\frac{3}{8}r_0s_1^2+\frac{1}{10}r_0r_1=0.
\end{align*}
By using the same tools than in the previous example we obtain that
one of the approximated solutions of the above system is $r_0\approx
0.4471066159$, $r_1\approx -0.0009814101$ and $s_1\approx
-0.0196567414$. We search simple rational approximations of $r_0$,
$r_1$ and $s_1$, doing again the respective continuous fraction
expansions and we obtain the candidate
$$
\tilde{r}(t)=\frac{1}{2}-\frac{1}{1018}\cos(t)-\frac{1}{50}\sin(t),
$$
to be an approximate periodic solution of~\eqref{2aa}. It can be
seen that it has accuracy $\widetilde{S}\approx 0.046$. Doing all
the computations needed to apply Theorem~\ref{T1} we obtain that we
are not under its hypotheses. Therefore we need to continue with the
HBM of second order.

Doing the second order approach we obtain five algebraic polynomial
equations, that we omit for the sake of simplicity. Unfortunately,
neither using the resultant method as in the previous cases, nor
using the more sophisticated tool of Gr\"{o}bner basis, our
computers are able to obtain an approximate solution to start our
theoretical analysis.

\smallskip

\noindent {\bf A numerical approach.} First, we search a numerical
solution of \eqref{2aa} by using the Taylor series method. From this
approximation we compute, again numerically, its first Fourier
terms, obtaining
$$
\tilde{r}(t)=\sum_{k=0}^3 r_{k}\cos(kt)+s_{k}\sin(kt),
$$
where
\begin{align*}
r_0=&\,\,0.4483561517,\qquad r_1=-0.0024133439,\quad s_1=-0.0193837572,\\
r_2=&-0.0037463296,\quad
s_2=-0.0220176517,\\r_3=&-0.0012390886,\quad s_3=0.0003784656.
\end{align*}
The accuracy of $\tilde{r}(t)$ is $0.00289$. If we take a new nicer
approximation, using again some convergents of $r_{k}$ and $s_k$, we
obtain {\small\begin{equation}\label{af2}
\bar{r}(t)=\frac{4}{9}-\frac{1}{693}\cos(t)-\frac{1}{51}\sin(t)
-\frac{1}{653}\cos(2t)-\frac{1}{45}\sin(2t)-\frac{1}{780}\cos(3t),
\end{equation}}
\!with accuracy $0.00298,$ quite similar to the one of
$\bar{r}_1(t)$. Note that~\eqref{af2} is precisely the approximation
of the periodic solution of~\eqref{2aa} stated in
Proposition~\ref{ppp}.

\begin{proof}[Proof of Proposition~\ref{ppp}] We already know that the accuracy of
$\bar r(t)$ is $S:=0.003.$ To apply Theorem~\ref{T1} we will compute
$M$ and~$K$.

First we calculate $A(t)=\int_0^t{\frac{\partial}{\partial
r}X(\bar{r}(t),t)}.$
$$
\begin{array}{ll}
A(t)=&\frac{2891685439}{72733752000}-\frac{347888350813299559}{1778094556332494400}t
-\frac{561179}{36756720}\cos(t)-\frac{685338551}{8000712720}\sin(t)\\\\&
-\frac{757058717}{48004276320}\cos(2t)
-\frac{40221206418131}{273447836421760}\sin(2t)
-\frac{2923231}{576974475}\cos(3t)\\\\&+\frac{37724429}{36003207240}\sin(3t)
-\frac{353400139}{96008552640}\cos(4t)
+\frac{17671001708653999}{42674269351979865600}\sin(4t)\\\\&
+\frac{5358811}{300026727000}\cos(5t)
+\frac{4708003}{20001781800}\sin(5t)
+\frac{1537}{207810720}\cos(6t)\\\\&+\frac{43551971479}{1438264594166400}\sin(6t)
+\frac{1}{327600}\cos(7t)
-\frac{1}{4753840}\sin(7t)\\\\&-\frac{1}{12979200}\sin(8t).
\end{array}
$$

Now, by using again Lemma~\ref{cota}, we find a deformation constant
$M$. In this case we use as lower bound for $A$ the piecewise
function $L$ formed by $7$ straight lines and $\ell=1/18$. We obtain
that we can take $M=7$. Therefore $2MS\approx 0.042$.

Since it can be seen that $0.4\leq\bar{r}(t)\leq 0.47$ in
Theorem~\ref{T1} we can consider the interval $I=[0.358, 0.512]$.

Then
$$
\max_{I\times[0,2\pi]}\left|\frac{\partial^2}{\partial
r^2}X(r,t)\right|\leq\frac{1}{5}+ 6||\bar r||_{_\infty}=\frac{1}{5}+
6(0.512)=3.272=:K
$$
Finally, $2M^2KS\approx 0.962<1$ and the first part of
Theorem~\ref{T1} applies. Hence equation~\eqref{2aa} has a periodic
solution $r^*(t)$ satisfying
\begin{equation}\label{boundppp}
||\bar{r}-r^*||_{_\infty}\leq 0.042, \end{equation} and is the only
one in this strip.

It can also be seen that
$$
\left|\int_0^{2\pi}\frac{\partial}{\partial
r}X(\bar{r}(t),t)dt\right|>1.2.
$$
Since $2\pi/M \approx 0.9$, the hyperbolicity of $r^*(t)$ follows
applying the second part of the theorem.
\end{proof}

Notice that the example of system~\eqref{ex2f} that we have studied
is $a=\lambda$ and $b=-\lambda$ with $\lambda=1/10.$ With the same
techniques it can be seen that the same function $\bar r(t)$ given
in the statement of Proposition~\ref{ppp} is an approximation of the
unique periodic orbit of the system when $|\lambda-1/10|<1/500,$
which also satisfies~\eqref{boundppp}.

\subsection*{Acknowledgements}
The authors are partially supported by a MCYT/FEDER grant number
MTM2008-03437 by a CIRIT grant number 2009SGR 410.

\end{document}